\documentclass{article}
\normalsize

\usepackage[cmex10]{amsmath}
\usepackage{amsfonts}
\usepackage{amsthm}

\newtheorem{defi}{Definition}

\newtheorem{thm}[defi]{Theorem}
\newtheorem{lem}[defi]{Lemma}
\newtheorem{cor}[defi]{Corollary}

\title{A family of  Alltop functions that are EA-inequivalent to the cubic function}

\begin{document}

%
%

%
%
%

\author{Joanne L. Hall\thanks{This research was done while J. L. Hall was with the Department of Algebra, Charles University Prague, Czech Republic. J. L. Hall is currently with the Queensland University of Technology.},~
        Asha Rao
\thanks{A. Rao is with RMIT University, Melbourne, Australia.  e-mail: asha@rmit.edu.au
},~
        and~Stephen~M.~Gagola~III\thanks{ S. M. Gagola III is with the Department
of Algebra, Charles University, Prague, Czech Republic.}
}

%



\maketitle

\begin{abstract}
Sequences with optimal correlation properties are much sought after for applications in communication systems.  In 1980,  Alltop (\emph{IEEE Trans. Inf. Theory} 26(3):350-354, 1980) described a set of sequences based on a cubic function and showed that these sequences were optimal with respect to the known bounds on auto and crosscorrelation.  Subsequently these sequences were used to construct mutually unbiased bases (MUBs), a structure of importance in quantum information theory.  The key feature of this cubic function is that its difference function is a planar function.  Functions with planar difference functions have been called \emph{Alltop functions}.  This paper provides a new family of Alltop functions and  establishes the use of Alltop functions for construction of sequence sets and MUBs.
\end{abstract}


keywords:
Planar function, Alltop function, mutually unbiased bases, MUBs, CDMA.

%

\section{Introduction}
%
%
%
%
Sequences with optimal correlation properties are much sought after for applications in communication systems. The root mean square (rms) correlation and the maximum correlation amplitudes of a sequence are the two correlations of interest in the design of code-division multiple access (CDMA) and such other systems \cite{DY2007}. There are known bounds on these correlations -  Welch's bounds \cite{Welch74} and Levenstein's bounds \cite{KabLev78} - and there is ongoing research into the design of sequences that meet these bounds \cite{ZF2011,Schmidt2011}.  For example it is known that no set of $N$ sequences of period $K$ can meet the Welch bound for maximum correlation if $N > K^2$ \cite{DY2007}. Thus, for $N > K^2$, sequences are considered optimal if they meet the Levenstein bounds.

In 1980, Alltop \cite{Alltop80} constructed complex periodic sequences which nearly meet the Welch bound for maximum correlation, using a cubic polynomial over the field $\mathbb{F}_p$ for $p$ prime, $p > 3$. This construction was extended \cite{KR03} to all prime power fields $\mathbb{F}_{p^r}$ for $p > 3$ and used to construct mutually unbiased bases (MUBs), structures of importance in quantum information theory \cite{WF89}. MUBs were originally constructed by Wooters and Fields in 1989 using quadratic functions over all prime power fields, with the construction being generalized to using planar functions \cite{RS07, DY2007} in 2007. In addition, Ding and Yin \cite{DY2007} also show that the associated complex periodic sequences meet the Levenstein bound for maximum correlation for $N > K^2$.

The difference function of the Alltop (cubic) polynomial is a planar function and consequently, functions whose difference functions are planar were called `Alltop functions' \cite{HRD12}. Coincidentally, the MUBs constructed from the cubic Alltop polynomial (\cite{KR03}) are equivalent to the set of MUBs constructed using a quadratic function, which is planar \cite{GR09}. The question that arises is whether there exist Alltop polynomials other than a cubic polynomial. Initial searches for Alltop functions have shown that several classes of planar functions are not the difference functions of any polynomial~\cite{HRD12}. In this paper we extend those results and give a family of Alltop functions whose difference functions are EA-equivalent to $x^2$, but which are themselves EA-inequivalent to $x^3$.

The paper is organized in the following manner: Section \ref{sec:prop} describes some of the properties of Alltop polynomials, highlighting some analogous properties to planar polynomials.  In Section~\ref{sec:new}, a family of Alltop functions is constructed. These Alltop functions are shown to be EA-inequivalent to $x^3$ even though their difference functions are EA-equivalent to $x^2$.

In Section \ref{sec:mubs} it is shown that any Alltop function also generates a complete set of MUBs while in Section \ref{sec:signals} it is shown that any Alltop function generates sets of sequences which are optimal with respect to the Levenstein  bound for maximum correlations.

\section{Properties of Alltop polynomials \label{sec:prop}}

Let $\mathbb{F}_{p^r}$ be a field of characteristic $p$. A function $f : \mathbb{F}_{p^r} \rightarrow \mathbb{F}_{p^r}$ is  called a \emph{planar
function} if for every  $a \in \mathbb{F}_{p^r}^*$ the function
$\Delta_{f,a}: x \mapsto f(a + x)- f (x)$ is a bijection \cite{CM97D}.  Planar functions do not exists when $p=2$ \cite{RS1989}. A function $A : \mathbb{F}_{p^r} \rightarrow \mathbb{F}_{p^r}$ is called an \emph{Alltop function} if $\Delta_{A,a}(x)$ is a planar function for all $a\in\mathbb{F}_{p^r}^*$ \cite{HRD12}. Here, as shown in \cite[Thm 15]{HRD12}, $p$ cannot equal 3, and hence $p \geq 5$.

Planar functions have been used to construct mutually unbiased bases \cite{KR03, RS07}  and sequences with low autocorrelation \cite{DY2007}.  The feature of planar functions that is used in these applications is the magnitude of the character sum.

\begin{thm}\cite[Thm 2.3]{CM97A}\label{thm:planarsum}
Let $\Pi(x)$ be a planar function on $\mathbb{F}_{p^r}$, $\omega=e^{2i\pi /p}$, and $\chi(x)=\omega^{ \textrm{ \em tr}(x)}$, then
\begin{equation}
\left|\sum_{x\in \mathbb{F}_{p^r}}\chi(a\Pi(x)+bx)\right|=\sqrt{p^r}
\end{equation}
for any $b\in {\mathbb{F}}_{p^r}$ and $a\in {\mathbb{F}}_{p^r}^*$.
\end{thm}

Another useful property of planar functions is that the planarity of the function is preserved when composed with a linearized (also called additive) polynomial \cite{CM97D}.

A function $L(x)\in\mathbb{F}_{p^r}[x]$ is called \emph{additive} if $L(x)+L(y)=L(x+y)$ for all $x,y\in\mathbb{F}_{p^r}$.   All additive functions on $\mathbb{F}_{p^r}$  have the shape $L(x)=\sum_{k=0}^{r-1} a_kx^{p^k}$, $a_k\in\mathbb{F}_{p^r}$ \cite{CM97D}.    Additive polynomials have a useful property with relation to difference functions:
\begin{lem}\cite[Lem 2.2]{CM97D}\label{lem:planaradd}
If $f, L\in\mathbb{F}_{p^r}[x]$ with $L(x)$ an additive polynomial, and $a\in\mathbb{F}_{p^r}^*$, then
\begin{equation}
\Delta_{f,L(a)}(L(x))=\Delta_{f\circ L,a}(x)
\end{equation}
\end{lem}

A simple calculation \cite{CM97D} shows that if $\Pi(x)$ is a planar polynomial, $L(x)$ an additive polynomial and $c\in\mathbb{F}_{p^r}$, then
\begin{equation}\label{eqn:planarcalc}
\Pi'(x)=\Pi(x)+L(x)+c
\end{equation} is also a planar polynomial.

 A Dembowski-Ostrom polynomial \cite{CM97D} is a polynomial $f(x)\in\mathbb{F}_{p^r}[x]$ with the shape
\[f(x)=\sum_{i,j=0}^{r-1}a_{ij}x^{p^i+p^j} \quad\quad\text{with}\;a_{ij}\in\mathbb{F}_{p^r}.\]
Most known planar functions are Dembowski Ostrom polynomials.

The Alltop analogue to equation (\ref{eqn:planarcalc}) is given below.

\begin{lem}\label{lem:EAalltop}
If $A(x)$ is an Alltop function, $\Pi(x)$ a Dembowski-Ostrom  planar function, $L(x)$ an additive function and $c\in\mathbb{F}_{p^r}$ then
\begin{equation}
A'(x)=A(x) + \Pi(x)+L(x)+c
\end{equation} is also an Alltop function.
\end{lem}
\begin{proof}
For $a\in\mathbb{F}_{p^r}^*$, 
\begin{align}
\Delta_{A',a}(x)=  & A'(x+a)-A'(x),\\
 = & A(x+a)+\Pi(x+a)+L(x+a)+c\nonumber\\
& -A(x)-\Pi(x)-L(x)-c,\\
= &\left(A(x+a)-A(x)\right)+\left(\Pi(x+a)-\Pi(x)\right)\nonumber\\
& +\left(L(x+a)-L(x)\right).
\end{align}
Note that $\left(A(x+a)-A(x)\right)$ is a planar function, $\left(\Pi(x+a)-\Pi(x)\right)$ is a linear permutation, and $\left(L(x+a)-L(x)\right)=L(a)$ is a constant.  Hence from equation \ref{eqn:planarcalc}, $\Delta_{A',a}(x)$ is a planar function for every  $a\in\mathbb{F}_{p^r}^*$, implying that $A'(x)$ is an Alltop function.

\end{proof}

\begin{lem} \cite[Lem 2.3]{CM97D}\label{lem:planarfacts}
Let $L(x)$ be an additive polynomial, then
the following are equivalent
\begin{enumerate}
\item[(i)]
$\Pi(L(x))$ is a planar polynomial.
\item[(ii)]
$L(\Pi(x))$ is a planar polynomial.
\item[(iii)]
$\Pi(x)$ is a planar polynomial and $L(x)$ is a permutation polynomial.
\end{enumerate}
\end{lem}
An additive polynomial, $L(x)$, is  a permutation polynomial over $\mathbb{F}_{p^r}$ if and only if $L(x)$ has no nonzero roots in $\mathbb{F}_{p^r}$ \cite[Thm 7.9]{LN97}.  An analogous result to Lemma \ref{lem:planarfacts} holds for Alltop functions:

\begin{lem}\label{lem:alltopfact}
Let $L(x)$ be an additive polynomial, then the following are equivalent.
\begin{enumerate}
\item[(i)]
$A(L(x))$ is an Alltop polynomial.
\item[(ii)]
$L(A(x))$ is an Alltop polynomial.
\item[(iii)]
$A(x)$ is an Alltop polynomial and $L(x)$ is a permutation polynomial.
\end{enumerate}
\end{lem}
\begin{proof}
Follows closely the proof of Lemma \ref{lem:planarfacts} as published  in \cite[Lem 2.3]{CM97D}.
Let  $a\in\mathbb{F}_{p^r}^*$, then
\begin{align}
\Delta_{L\circ A,a}(x)= & L(A(x+a))-L(A(x))\nonumber\\
 = & L(A(x+a)-A(x))\nonumber\\
 = & L(\Delta_{A,a}(x)).\label{eqn:L}
 \end{align}
If $L(A(x))$ is an Alltop polynomial, then $\Delta_{L\circ A,a}(x)$ is a planar polynomial, and hence from equation (\ref{eqn:L}),  $L(\Delta_{A,a}(x))$ is a planar polynomial.  From Lemma \ref{lem:planarfacts}, $L(x)$ is a permutation polynomial and $\Delta_{A,a}(x)$ is a planar polynomial, and hence $A(x)$ is an Alltop polynomial.  Thus \emph{(ii)}$\Rightarrow$\emph{(iii)}.  Conversely if \emph{(iii)} holds then equation (\ref{eqn:L}) shows that $L(A(x))$ is an Alltop polynomial, and hence \emph{(iii)}$\Rightarrow$\emph{(ii)}.

To show that \emph{(i)}$\Rightarrow$\emph{(iii)} consider $A(L(x))$, and suppose that $\Delta_{A\circ L,a}(x)$ is a planar polynomial for all $a\in\mathbb{F}_{p^r}^*$.  Using Lemma \ref{lem:planaradd} $\Delta_{A,L(a)}(L(x))$ is also a planar polynomial. Since $L(x)$ is an additive polynomial, from  Lemma \ref{lem:planarfacts},  $\Delta_{A\circ L,a}(x)$ is a planar polynomial and $L(x)$ is a permutation polynomial.  This holds for all $a\in\mathbb{F}_{p^r}^*$, thus $A(x)$ is an Alltop polynomial.
  The converse result, \emph{(iii)}$\Rightarrow$\emph{(i)}, follows immediately from Lemma \ref{lem:planaradd}.
\end{proof}

Lemma \ref{lem:planarfacts} leads to the definition of extended affine  equivalence for planar functions.  An \emph{affine} function is an additive function plus a constant. Two functions $f_1,f_2$ are \emph{extended affine} (EA) equivalent if there exist affine functions $l_1,l_2,l_3$ such that $f_1(x)=l_1\circ f_2\circ l_2(x) +l_3(x)$ where $l_1,l_2$ are permutations.  Using geometric arguments Coulter and Matthews \cite[\S 5]{CM97D} show that any function EA-equivalent to a planar function, is a planar function, and thus the set of planar functions may be partitioned into EA-equivalence classes.  The following is the analogue for Alltop functions.
\begin{lem}
Let $l_1,l_2,l_3$ be affine functions with $l_1(x)$ and $l_2(x)$ permutations, and $d(x)$ a Dembowski-Ostrom polynomial.  If $A(x)$ is an Alltop polynomial on $\mathbb{F}_{p^r}$ and
\begin{equation}
A'(x)=l_1\circ A\circ l_2(x)+d(x)+l_3(x)
\end{equation}
 then $A'(x)$ is an Alltop polynomial, and  there exists $b\in\mathbb{F}_{p^{r}}^*$ such that for $a\in\mathbb{F}_{p^{r}}^*,  \Delta_{A,a}(x)$ is EA-equivalent to $\Delta_{A',b}(x)$.
\end{lem}

\begin{proof}
 Let $l_1(x)=L_1(x)+c_1$ and $l_2(x)=L_2(x)+c_2$ where $L_1(x), L_2(x)$ are  additive functions, and $c_1,c_2\in\mathbb{F}_{p^{r}}$. For $b\in\mathbb{F}_{p^{r}}^*$,
\begin{align}
\Delta_{l_1\circ A,b}(x) = & l_1\circ A(x+b)-l_1\circ A(x)\\
= & L_1\circ A(x+b)+c_1-L_1\circ A(b)-c_1\\
= & L_1\circ \Delta_{A,b}(x)\label{eqn:daa}
\end{align}
which is similar to equation \ref{eqn:L}.
\begin{align}
\Delta_{A\circ l_2,b}(x)= & A\circ l_2(x+b)-A\circ l_2(x)\\
= & A(L_2(x+b)+c_2)-A(L_2(x)+c_2)\\
= & A(L_2(x)+L_2(b)+c_2)-A(L_2(x)+c_2)\\
= & A(l_2(x)+L_2(b))-A(l_2(x))\\
=& \Delta_{A,L_2(b)}(l_2(x))\label{eqn:dala}
\end{align}
Let $g(x)=d(x)+l_3(x)$ then $\Delta_{g,b}(x)$ is an affine function.  Using equations (\ref{eqn:daa}, \ref{eqn:dala}),
\begin{align}
\Delta_{A',b}(x)= & L_1\circ \Delta_{A,L_2(b)}(l_2(x))+\Delta_{g,b}(x).
\end{align}
By taking $a=L_2(b)$ then $a \neq 0$ and   $\Delta_{A,a}(x)$ is EA-equivalent to $\Delta_{A',b}(x)$.
\end{proof}

Note that for $d(x)\neq 0$,  $A(x)$ and  $A'(x)$ are EA-inequivalent yet $\Delta_{A,a}(x)$ and $\Delta_{A',a}(x)$ are EA-equivalent for all $a\in\mathbb{F}_{p^r}^*$ such that  $L_2(a)=a$.  Thus an EA-equivalence class of planar functions may contain several EA-equivalence classes of Alltop functions.

\begin{thm}\cite{KO2011}
 Let $L_1(x),L_2(x)\in\mathbb{F}_{p^r}[x]$ be additive polynomials where $p>2$.  If  $\Pi(x)=L_1(x)L_2(x)$  is a planar polynomial then $L_1(x)$ and $L_2(x)$ are both permutation polynomials.
\end{thm}
An analogous result also holds for Alltop functions:
\begin{thm}\label{thm:alltopadditive}
Let $L_1(x),L_2(x),L_3(x)\in\mathbb{F}_{p^r}[x]$ be additive polynomials where $p>3$.  If $A(x)=L_1(x)L_2(x)L_3(x)$ is an Alltop polynomial then $L_1(x),L_2(x),L_3(x)$ are all permutation polynomials.
\end{thm}
\begin{proof}
For $a\in\mathbb{F}_{p^{r}}^*$, the difference function of $A(x)$ is
\begin{align}
\Delta_{A,a}(x)= & L_1(x)L_2(x)L_3(a) + L_1(x)L_2(a)L_3(x)\nonumber\\
  & + L_1(a)L_2(x)L_3(x) + L_a(x)+ c
\end{align}
where $L_a(x)$ is an additive polynomial, and $c\in\mathbb{F}_{p^r}$.
Let
\begin{align}
\Pi_a(x)= & L_1(x)L_2(x)L_3(a) + L_1(x)L_2(a)L_3(x) \nonumber\\
& + L_1(a)L_2(x)L_3(x),
\end{align}
 then $\Delta_{A,a}(x)$ is planar if and only if $\Pi_a(x)$ is planar for all $a\in\mathbb{F}_{p^r}^*$. Further, for $b\in\mathbb{F}_{p^{r}}^*$,
\begin{align}
\Delta_{\Pi_a,b}(x)=  & L_1(x)[L_2(a)L_3(b)+L_2(b)L_3(a)]\nonumber\\
& + L_2(x)[L_1(a)L_3(b)+L_1(b)L_3(a)] \nonumber \\
& +L_3(x)[L_1(a)L_2(b)+L_1(b)L_2(a)]+c
\end{align}
where $c\in \mathbb{F}_{p^r}$.

Let
$M_{ab}(x)=\Delta_{\Pi_a,b}-c$.
Then $M_{ab}(x)$ is an additive polynomial, and is thus a permutation polynomial if and only if its only root is $0$ \cite[Thm 7.9]{LN97}.
Putting $a=b$ gives
\begin{align}
M_{aa}(x)= &  7 2L_1(x)L_2(a)L_3(a)+ 2L_1(a)L_2(x)L_3(a)\nonumber\\
&  7 +2L_1(a)L_2(a)L_3(x).
\end{align}
Assume that there exists $a\in\mathbb{F}_{p^r}^*$ such that $L_1(a)=0$ or $L_2(a)=0$ or $L_3(a)=0$.  In each case $M_{aa}(a)=0$, and hence is not a permutation polynomial.
\end{proof}

\underline{Note}: The condition of Theorem \ref{thm:alltopadditive} is necessary but not sufficient.  For example let $L_1(x)=x^p$, $L_2(x)=x$ and $L_3(x)=x$ in $\mathbb{F}_{p^2}$.  Through simple but tedious calculations it is noted that

\[\Delta\Delta_{A,a,b}(x)=2axb(x^{p-1}+a^{p-1}+b^{p-1}) +c_{ab}\]
for $a,b \neq 0$ and $c_{ab}\in\mathbb{F}_{p^2}$. If $p\equiv 2(\!\!\!\!\mod 3)$, then \mbox{$|\mathbb{F}_{p^2}^*|\equiv 0(\!\!\!\!\mod 3)$}, and hence $\mathbb{F}_{p^2}^*$ contains a subgroup of order $3$. Let $j$ be a generator of this subgroup, then $j^{p-1}=j$. Let $x=j^2$, $a=j$, and $b=1$, then $x^{p-1}+a^{p-1}+b^{p-1}=0$, and hence $\Delta\Delta_{A,a,b}(x)$ is not a permutation. Thus if \mbox{$p\equiv 2(\!\!\!\!\mod 3)$}, then $A(x)=x^{p+2}$ is not an Alltop function on $\mathbb{F}_{p^2}$.

\section{A new  Alltop polynomial \label{sec:new}}
Until now, the only example of an Alltop function was $x^3$.  We now introduce a family of Alltop polynomials along with the corresponding planar function.
\begin{thm}\label{thm:p^r+2} Let $p\geq 5$ be an odd prime and  $r$ an integer such that $3$ does not divide $p^{r}+1$.  Then
$A(x)=x^{p^r+2}$
is an   Alltop polynomial on $\mathbb{F}_{p^{2r}}$.
\end{thm}

In order to show that $A(x)$ is an Alltop function, a planar function is needed.  Let
\[D_f(x,y)=f(x+y)-f(x)-f(y).\]
If $f$ is a planar function, then $D_f(x,y)$ is a bijection in $x$ for all nonzero $y$.  The definition of a planar function may be given using the function $D_f$ rather than $\Delta_{f,a}$ \cite{CH2008}.  An advantage of $D_f$ is that if $f$ is a  Dembowski-Ostrom polynomial, then  $D_f(x,y)\neq 0$ for all $x,y\neq 0$ if and only if $f$ is planar.

\begin{proof}
For $a\in\mathbb{F}_{p^{r}}^*$, the difference function of $A(x)$ is
 \begin{align}
\Delta_{A,a}(x) &= 2ax^{p^r+1}+a^2x^{p^r}+a^{p^r}x^2+2a^{p^r+1}x+a^{p^r+2}.
 \end{align}
Any planar function $f$ is EA-equivalent to $f+L+c$ where $L$ is a linear function and $c$ is a constant.
Hence, for any $a\in\mathbb{F}_{p^{2r}}^*$, $\Delta_{A,a}(x)$ is EA-equivalent to
\[ \Pi_a(x)=2ax^{p^r+1}+a^{p^r}x^2\]
and
\begin{align}
D_{\Pi_{a}}(x,b)= \;& 2a x^{p^r}b+2axb^{p^r}+2a^{p^r}xb.
\end{align}
To show that $\Pi_a$ is a planar function, it must be shown that $D_{\Pi_a}(x,b)=0$ if and only if at least one of  $x$ or $b$ is zero.
Suppose both $x$ and $ b$ are non-zero and 
\begin{align}
& & D_{\Pi_a}(x,b) & =0\\
\Longrightarrow & &  2a x^{p^r}b+2axb^{p^r}+2a^{p^r}xb & =0\\
\Longrightarrow & &  x^{p^r}b+xb^{p^r}+a^{p^r-1}xb & =0 \label{eqn:D}\\
\Longrightarrow & &
 \left(x^{p^r}b+xb^{p^r}+a^{p^r-1}xb\right)^{p^r} & =0 \\
\Longrightarrow & &  x^{p^r}b+xb^{p^r}+a^{1-p^r}x^{p^r}b^{p^r} & =0 \label{eqn:Dq}
\end{align}
By subtracting equation  (\ref{eqn:Dq}) from (\ref{eqn:D}) it follows that
\begin{align}
 & & a^{p^r-1}xb-a^{1-p^r}x^{p^r}b^{p^r} & =0,\\
\Longrightarrow & &  1-\left(\frac{xb}{a^2}\right)^{p^r-1} & =0.
\end{align}
Thus, there exists an element $c\in \mathbb{F}_{p^r}^*$ such that $xb=ca^2$.  From equation (\ref{eqn:D})
\begin{align}
& & xb\left(x^{p^r-1}+b^{p^r-1}+a^{p^r-1}\right) &=0,\\
\Longrightarrow & & ca^2\left[\left(\frac{c}{b}a^2\right)^{p^r-1}+b^{p^r-1}+a^{p^r-1}\right] & =0.
\end{align}
Since $c\in\mathbb{F}_{p^r}^*$, it is known that $c^{p^r-1}=1$ and hence
\begin{align}
& & \left(\frac{a^2}{b}\right)^{p^r-1}+b^{p^r-1}+a^{p^r-1} & =0\\
\Longrightarrow & & \left(\left(\frac{a}{b}\right)^{p^r-1}\right)^2+\left(\frac{a}{b}\right)^{(p^r-1)}+1 & =0\label{eqn:ab1}\\
\Longrightarrow & & \left(\left(\frac{a}{b}\right)^{p^r-1}\right)^3-1 & =0 \label{eqn:ab2}
\end{align}
Since $p\neq 3$, from equation (\ref{eqn:ab1}), $(a/b)^{p^r-1}\neq 1$.  Hence, equation (\ref{eqn:ab2}) has solutions if and only if $3$ divides $\frac{p^{2r}-1}{p^r-1}=p^r+1$.
\end{proof}

Note:
Let $1$ and $0$  be the multiplicative and additive identity elements of $\mathbb{F}_{p^{2r}}$ respectively with the conditions on $p$ and $r$  as in Theorem \ref{thm:p^r+2}. Let $3=1+1+1$ and   let $(-3)$ be the element of $\mathbb{F}_{p^{2r}}$  for which $(-3)+3=0$.    Thus $(-3)\in\mathbb{F}_p$, the subfield of order $p$.   Now let $\alpha$ be a generator of the cyclic group $\mathbb{F}_{p^{2r}}^*$ and let $\beta=\alpha^{(p^r+1)/2}$.  Thus $\beta$ is not in the subfield of order $p$ whereas $\beta^2$ is. Furthermore $\beta^2$ is a generator of $\mathbb{F}_{p}^*$, the cyclic group of order $p-1$.  Hence, there exists an integer $n$ such that $(\beta^2)^n=-3$.  Therefore, if $\gamma=\beta^n$ then $\gamma^2=-3$ and
\begin{align}
& a^{p^r}\Big(2ax^{p^r+1}+a^{p^r}x^2\Big)+\frac{-1+\gamma}{2}a\Big(2ax^{p^r+1}+a^{p^r}x^2\Big)^{p^r}\nonumber\\
& =\left(a^{p^r}x+\frac{1+\gamma}{2}ax^{p^r}\right)^2\!.
\end{align}
Hence $\Pi_a(x)$ is EA-equivalent to $x^2$ for all $a\in\mathbb{F}_{p^{2r}}^*$.

We use an extension of \cite[Prop 1]{BH2011C} to show that \mbox{$A(x)=x^{p^r+2}$} is not EA-equivalent to $x^3$.
\begin{lem}\label{lem:equivx^3}
Let $p$ be an odd prime and $n$ a positive integer.  Any function $f$ over $\mathbb{F}_{p^n}$ of the form
\[f(x)=\sum_{0\leq k,j,i<n}a_{kji}x^{p^k+p^j+p^i}\]
such that $a_{kji}=0$ for  $k=j=i$,  is EA-inequivalent to $x^3$.
\end{lem}
\begin{proof}
If $f(x)$ is EA-equivalent to $x^3$ then there exist affine functions $l_1(x),l_2(x)$ and $l_3(x)$ such that
\[ l_1\circ (l_2(x))^3+l_3(x)=f(x).\]
Let $l_1(x)=L_1(x)+c$ and $l_2(x)=L_2(x)+d=\sum_{i=0}^{n-1}d_ix^{p^i}+d$, where $L_1(x)$ and $L_2(x)$ are additive functions and $c, d$ are constants.  Then
\begin{align}
l_1\circ (l_2(x))^3+l_3(x) = \; &  l_1\left(\sum_{i=0}^{n-1}d_i^3x^{3p^i}+ M(x)\right)+l_3(x),\nonumber\\
=\; &  L_1\left(\sum_{i=0}^{n-1}d_i^3x^{3p^i}\right)+ L_1\circ M(x)\nonumber\\
 & +c+l_3(x), \label{eqn:M}
\end{align}
where $M(x)$ contains terms which are not of the form $x^{3p^i}$.  Since $f(x)$ contains no term of the form $x^{3p^i}$,  the right hand side of equation \ref{eqn:M} must have no terms of the form  $x^{3p^i}$.  The functions $L_1\circ M(x)$ and $l_3(x)$ have no terms of the form  $x^{3p^i}$, therefore it is also required that $ L_1\left(\sum_{i=0}^{n-1}d_i^3x^{3p^i}\right)$ has no terms of the form  $x^{3p^i}$.  This therefore requires that $d_i=0$ for all $0\leq i\leq n-1$.   Hence $L_2=0$, that is $l_2(x)$ is not a permutation, and hence $f(x)$ is not EA-equivalent to $x^3$.
\end{proof}
As a corollary $A(x)=x^{p^r+2}$ is not EA-equivalent to $x^3$.

\section{Mutually unbiased bases \label{sec:mubs}}

Alltop's original construction of sequences was published in 1980 \cite{Alltop80}.  Some years later it was noted that this construction could also be used to construct mutually unbiased bases \cite{KR03}, a structure of importance in quantum information theory \cite{BB84, SARG04, WF89}, and signal set design \cite{DY2007}.

Two orthonormal bases $B_1$ and $B_2$ of  $\mathbb{C}^d$ are \emph{unbiased} if $|\langle \vec{x}|\vec{y}\rangle|=\frac{1}{\sqrt{d}}$ for all $\vec{x}\in B_1$ and $\vec{y}\in B_2$. A set of bases of $\mathbb{C}^d$ which are pairwise unbiased is a set of \emph{mutually unbiased bases} (MUBs).
This idea arose in a quantum physics context \cite{Schw60} when it was noted that
a quantum system prepared in a basis state from $B_1$ reveals no
information when measured with respect to the basis $B_2$.   There can be a maximum of $d+1$ MUBs in $\mathbb{C}^d$ and such sets are called complete. There are constructions of complete sets of MUBs in all prime power dimensions \cite{KR03, WF89}. However it is unknown if complete sets exist when $d$ is not a prime power \cite{SPR04}.

 The following is a construction of MUBs in odd prime powers which uses planar functions.  Let $\omega_p=e^{i\pi/p}$ and $\chi(x)=\omega_p^{\text{tr}(x)}$.
  \begin{thm}[Planar function construction]\label{royscott}\cite[Thm 4.1]{RS07}\cite[Thm 4]{DY2007}
Let $\mathbb{F}_q$ be a field of odd characteristic $p$ and $\Pi(x)$ a planar function on $\mathbb{F}_q$.  Let $V_a:=\{\vec{v}_{ab}:b\in \mathbb{F}_q\}$ be the set of vectors
\begin{equation}\label{planarv}
\vec{v}_{ab}= \frac{1}{\sqrt{q}}\left(\omega_p^{\text{\emph{tr}}(a\Pi(x)+bx)}\right)_{x\in\mathbb{F}_q}= \frac{1}{\sqrt{q}}\Big(\chi\left(a\Pi(x)+bx\right)\Big)_{x\in\mathbb{F}_q}\quad
\end{equation}
 with $a,b\in\mathbb{F}_q$. The standard basis $E$ along with the sets $V_a$, $a\in \mathbb{F}_q$, form a complete set of $q+1$ MUBs  in $\mathbb{C}^q$.
\end{thm}

Alltop's original construction only works for $\mathbb{C}^q$ where $q\geq 5$ is a prime \cite{Alltop80}.  A generalization to prime powers was done by Klappeneker and Roettler in 2003 \cite{KR03}.

\begin{thm}[Alltop's Construction]\cite[Thm 1]{KR03}\label{alltop}
Let $\mathbb{F}_q$ be a field of odd characteristic $p\geq5$.  Let $V_a:=\{\vec{v}_{ab}:b\in \mathbb{F}_q\}$ be the set of vectors
\begin{align}\label{Alltopv}
\vec{v}_{ab}:= &  \frac{1}{\sqrt{q}}\left(\omega_p^{\text{\emph{tr}}((x+a)^3+b(x+a))}\right)_{x\in\mathbb{F}_q}\nonumber\\
=  & \frac{1}{\sqrt{q}}\Big(\chi\left((x+a)^3+b(x+a)\right)\Big)_{x\in\mathbb{F}_q}\quad
\end{align}
  with $a,b\in\mathbb{F}_q$. The standard basis $E$ along with the sets $V_a$, $a\in \mathbb{F}_q$, form a complete set of $q+1$ MUBs  in $\mathbb{C}^q$.
 \end{thm}

We now show that a complete set of MUBs can be constructed from any Alltop function over a field of odd characteristic $p \geq 5$.
\begin{thm}\label{alltopMUBs}
Let $\mathbb{F}_q$ be a field of odd characteristic $p\geq5$ and $A(x)$  an Alltop function over $\mathbb{F}_q$.  Let $V_a:=\{\vec{v}_{ab}:b\in \mathbb{F}_q\}$ be the set of vectors
\begin{align}\label{eqn:generalAlltopv}
\vec{v}_{ab}:= &  \frac{1}{\sqrt{q}}\left(\omega_p^{\text{\emph{tr}}(A(x+a)+b(x+a))}\right)_{x\in\mathbb{F}_q}\nonumber\\
= &  \frac{1}{\sqrt{q}}\Big(\chi\left(A(x+a)+b(x+a)\right)\Big)_{x\in\mathbb{F}_q}
\end{align}
  with $a,b\in\mathbb{F}_q$. The standard basis $E$ along with the sets $V_a$, $a\in \mathbb{F}_q$, form a complete set of $q+1$ MUBs  in $\mathbb{C}^q$.
 \end{thm}
\begin{proof}
\begin{align}\label{eqn:ip}
\langle\vec{v}_{ab}|\vec{v}_{cd}\rangle= & \frac{1}{q}\sum_{x\in\mathbb{F}_q}\chi\Big(A(x+a)-A(x+c)\nonumber\\
&\quad\quad\quad\quad +b(x+a)-d(x+c)\Big)
\end{align}
Let $y=x+c$ and $\alpha=a-c$, then
\begin{align}
\langle\vec{v}_{ab}|\vec{v}_{cd}\rangle= & \frac{1}{q}\sum_{x\in \mathbb{F}_q}\chi\Big(A(y+\alpha)-A(y)+(b-d)y+b\alpha \Big),\nonumber\\
= & \frac{1}{q}\sum_{x\in\mathbb{F}_q}\chi\Big(\Delta_{A,\alpha}(y)+(b-d)y+\alpha \Big).\label{eqn:alltopplanar}
\end{align}
From the definition of an Alltop type polynomial, we know that $\Delta_{A,\alpha}(y)$ is a planar polynomial, for $\alpha \neq 0$.
 By Lemma \ref{lem:planarfacts},  $\Delta_{A,\alpha}(x+c)+(b-d)(x+c)+b(a-c)$ is also a planar polynomial.  Then using Theorem \ref{thm:planarsum} we see that $|\langle\vec{v}_{ab}|\vec{v}_{cd}\rangle|=\frac{1}{\sqrt{q}}$ for $a\neq c$. 

When $a = c$, using \cite[Thm 5.4]{LN97}
\begin{align}
\langle\vec{v}_{ab}|\vec{v}_{ad}\rangle= & \frac{1}{q}\sum_{x\in\mathbb{F}_q}\chi\Big((b-d)(x+a)\Big)=0
\end{align}
 for $b \neq d$ and
\begin{align}
\langle\vec{v}_{ab}|\vec{v}_{ab}\rangle= & \frac{1}{q}\sum_{x\in\mathbb{F}_q}\chi\Big(0\Big)=1,
\end{align}
completing the orthonormal requirements for each of the bases.  Each component of each vector $\vec{v}_{ab}$ has magnitude $\frac{1}{\sqrt{q}}$, and is hence unbiased to the standard basis.  A complete set of MUBs in $\mathbb{C}^q$ has been generated.
\end{proof}

It has been shown \cite{GR09} that the set of MUBs generated using the Alltop function $A(x)=x^3$ is equivalent to the set of MUBs generated using the planar function $\Pi(x)=x^2$.  Is this equivalence  true in general?  A proof of the equivalence (or non-equivalence) of sets of Alltop MUBs and planar MUBs would require explicit calculation of character sums.  Unfortunately, explicit character sums are known in very limited cases \cite[5.30]{LN97}, insufficient to determine if the Alltop MUBs generated by $x^{p^r+2}$ are equivalent to the MUBs generated by $x^2$.

\section{Signal Sets  \label{sec:signals}}

Alltop's original paper \cite{Alltop80} was concerned with constructing sequences with low  correlation for use as signals in telecommunication devices.  Such sequences have important applications in communications technologies such as code-division multiple access (CDMA) and Direct Sequence CDMA (DS-CDMA).

Let $C=\{\vec{v}_i:0\leq i\leq N-1\}\subset \mathbb{C}^K$ be a set of normed complex vectors.  This is called an $(N,K)$ \emph{signal set}. The \emph{root-mean square} and \emph{maximum correlation amplitudes} are defined as
\begin{equation}
I_\textrm{rms}(C):=\sqrt{\frac{1}{N(N-1)}\sum_{\begin{smallmatrix}0\leq i,j\leq N-1\\ i\neq j \end{smallmatrix}}|\langle \vec{v}_i|\vec{v}_j\rangle|^2}
\end{equation}
\begin{equation}
I_\textrm{max}(C):= \max_{0\leq i<j\leq N-1}|\langle \vec{v}_i|\vec{v}_j\rangle|.
\end{equation}

CDMA and DS-CDMA  use signal sets to distinguish between the signals of different users.  $(N,K)$ signal sets with $N>K$ which are optimal with respect to $I_{\textrm{max}}$ are desirable for reducing interference when the number of users is greater than the signal space dimension \cite{DY2007}. Such signal sets could be used in future communications networks for increasing the user capacity within a limited frequency spectrum.

There are some lower bounds on $I_\textrm{rms}$ and $I_\textrm{max}$.  The Welch bounds \cite{Welch74}
\begin{equation}
I_\textrm{rms}\geq \sqrt{\frac{N-K}{(N-1)K}}
\end{equation}
\begin{equation}
I_{\textrm{max}}\geq \sqrt{\frac{N-K}{(N-1)K}}
\end{equation}
and the Levenstein bound \cite{KabLev78} for $N>K^2$
\begin{equation}\label{eqn:lev}
I_{\textrm{max}}\geq \sqrt{\frac{2N-K^2-K}{(K+1)(N-K)}}.
\end{equation}

It has been shown that a complete set of MUBs in $\mathbb{C}^K$, when considered as a $(K^2+K, K)$ signal set,  meets the rms Welch bound \cite{KR05}.  For $N>K^2$ the max Levenstein bound is tighter than the max Welch bound.

The construction of low correlation sequences and constructions of MUBs have continued in close tandem with  Ding and Yin \cite{DY2007} using planar functions to construct signal sets as well as MUBs at around the same time as Roy and Scott \cite{RS07}, the later using the construction of MUBs in the context of quantum physics.

\begin{thm}\cite[Thm 4]{DY2007}
Let $\Pi(x)$ be a planar function on $\mathbb{F}_q$.  Let
\begin{equation}
\vec{c}_{ab}= \frac{1}{\sqrt{q}}\left(\omega_p^{\text{\emph{tr}}(a\Pi(x)+bx)}\right)_{x\in\mathbb{F}_q}.
\end{equation}
Let $C_\Pi=\{\vec{c}_{ab}:a,b\in\mathbb{F}_q\}\cup E$.  Then $C_\Pi$ is a $(q^2+q,q)$ signal set with $I_{\emph{\textrm{max}}}=\frac{1}{\sqrt{q}}$.
\end{thm}

It has been noted that the signal sets from the planar function construction are optimal with respect to the Levenstein bound \cite{DY2007}: this is true in general for signal sets which are MUBs.
\begin{lem}\label{lem:levbound}
A complete set of MUBs in $\mathbb{C}^K$ when considered as a $(K^2+K,K)$ signal set meets the maximum Levenstein Bound.
\end{lem}
\begin{proof}
\begin{equation}
\sqrt{\frac{2(K^2+K)-K^2-K}{(K+1)(K^2+K-K)}}= \sqrt{\frac{1}{K}}
\end{equation}
\end{proof}
Thus signal sets which correspond to MUBs are optimal in both rms and max correlation amplitudes.  In particular, signal sets generated from the Alltop construction of MUBs (Theorem~\ref{alltopMUBs}) are also  optimal signal sets.
\begin{cor}
Let $A(x)$ be an Alltop type function on $\mathbb{F}_q$.  Let
\begin{equation}
\vec{v}_{ab}= \frac{1}{\sqrt{q}}\left(\omega_p^{\text{\emph{tr}}(A(x+a)+b(x+a)}\right)_{x\in\mathbb{F}_q}.
\end{equation}
Let $C_A=\{\vec{v}_{ab}:a,b\in\mathbb{F}_q\}\cup E$.  Then $C_A$ is a $(q^2+q,q)$ signal set with $I_{\emph{\textrm{max}}}=\frac{1}{\sqrt{q}}$
\end{cor}

\section{Conclusion}

The construction of MUBs and signal sets with optimal correlation, using a cubic function, has been known  for some time \cite{Alltop80, KR03}.  Along with a new family of, what we call, Alltop functions, we have expanded on this idea by showing that MUBs and optimal signal sets can be constructed by any Alltop function.  These applications of Alltop functions motivate further searches for Alltop functions.

There is an open problem regarding the possibility of non-equivalent Alltop functions generating non-equivalent MUBs.  Another open problem is analyzing the properties of Alltop signals with regards to other correlation measures and transmission methods.


%



\section*{Acknowledgment}
The authors would like to thank  Diane Donovan for useful discussions and  Cunshen Ding for highlighting reference \cite{DY2007}, and the anonymous reviewers whose suggestions greatly improved the readability of this paper, and in particular for providing a shorter proof for Lemma \ref{lem:alltopfact}.

\end{document}